\documentclass[]{article}

\RequirePackage[numbers]{natbib}
\usepackage{tikz,pgffor,pgfplots}
\usetikzlibrary[calc,intersections,through,backgrounds,patterns]
\usepackage{amsfonts,amssymb,float,amsmath,color,amsthm,eufrak,latexsym}
\usepackage{graphicx,epsfig,color,bm}
\usepackage{longtable, enumerate}
\usepackage{float}

\restylefloat{figure}

\newtheorem{remark}{Remark}
\newtheorem{thm}{Theorem}
\newtheorem{lemma}{Lemma}
\newtheorem{proposition}{Proposition}
\newtheorem{corollary}{Corollary}

\allowdisplaybreaks 

\newcommand{\ind}{\mathbb I}
\newcommand{\real}{\mathbb R}
\newcommand{\bola}[2]{\mathcal B (#1,#2)}
\newcommand{\norm}[1]{\left\Vert#1\right\Vert}
\newcommand{\abs}[1]{\left\vert#1\right\vert}
\newcommand{\esp}[2]{\mathbb{E}_{#1}\hspace{-0.03cm}\left(#2\right)}
\newcommand{\prob}[2]{\mathbb{P}_{#1}\hspace{-0.03cm}\left(#2\right)}
\newcommand{\var}[2]{\mbox{var}_{#1}\hspace{-0.03cm}\left(#2\right)}
\newcommand{\Cn}{\mathcal{C}_n}  
\newcommand{\Dn}{\mathcal{D}_n}  
\newcommand{\etahat}{\widehat{\eta}}
\newcommand{\muu}[1]{\mu\big(#1\big)}
\newcommand{\pa}[1]{\left(#1\right)}
\newcommand{\llave}[1]{\left\{#1\right\}}
\newcommand{\sop}[1]{\textrm{supp}\left(#1\right)}
\newcommand{\X}{\mathcal{X}}

\newcommand{\Y}{\mathcal{Y}}
\newcommand{\Hi}{\mathcal{H}}

\newcommand{\e}{\bm{e}}
\newcommand{\xf}{x}

\begin{document}

\begin{center}
		\Large \bf  \textcolor{black}{Nonparametric regression based on discretely sampled curves}
\end{center}

\begin{center}
	\bf Liliana Forzani$^\dag$, Ricardo Fraiman$^*$, and Pamela Llop$^{\dag}$ 
\end{center}

\begin{center}
	$^*$  Centro de Matem\'atica, Facultad de Ciencias (UdelaR), Uruguay.\\
	$^{\dag}$ Facultad de Ingenier\'ia Qu\'imica, UNL and researchers of CONICET, Argentina.
\end{center}

\begin{abstract}
In the context of nonparametric regression, we study conditions under which the consistency (and rates of convergence) of  estimators built from discretely sampled curves can be derived from the consistency of estimators based on the unobserved  whole trajectories. As a consequence, we derive asymptotic results for most of the regularization techniques used in functional  data analysis, including smoothing and basis representation. 
\end{abstract}


\section{Introduction}

Technological  progress in collecting and storing data provides datasets recorded at finite grids
of points that become denser and denser over time. Although in practice data always comes in the 
form of finite dimensional vectors, from the theoretical point of view, the classic multivariate 
techniques are not well suited to deal with data which, essentially, is infinite dimensional and 
whose observations within the same curve are highly correlated.

From a practical point of view, a commonly used technique to treat this kind of data is to 
transform the (observed) discrete values into a {function} via smoothing or a series approximations (see \cite{Bigot_2006}, \cite{Kneip_and_Ramsay_2008}, \cite{Ramsay_and_Silverman997, Ramsay_and_Silverman002, Ramsay_and_Silverman005}, or chapter 9
of \cite{oxford} and the references therein). For the analysis, we
can use the intrinsic infinite dimensional nature of the data and assume the existence of
continuous underlying stochastic processes which are observed ideally at every point. In this
context, the theoretical analysis is performed on the functional space where they take 
values (see
\cite{Ferraty_and_vieu_2006}). In what follows, we will refer to this last setting as the \it full
model\rm.  

Nonparametric regression is an important tool in functional data analysis (FDA) which has received considerable attention from 
different authors in both settings. For the full model, consistency results have been obtained by, among 
others, \cite{abraham_biau_cadre_2006}, \cite{Biau_Bunea_Wegkamp_2005}, 
\cite{Biau_Cerou_Guyader_2010}, \cite{Burba_Ferraty_Vieu_2009}, \cite{cerou_and_guyader_2006}, 
\cite{Ferraty_and_vieu_2006}, \cite{lian011}, and \cite{muller_2011}. In particular,  
\cite{forzani_fraiman_llop_2012} (see also the Corrigendum \cite{forzani_fraiman_llop_2014}) prove  a consistency 
result close to universality for the kernel (with random bandwidth) estimator. The first 
contribution of the present paper will be to prove the consistency of the $k$-nearest neighbor with kernel 
regression estimator (Proposition \ref{consistencia_knn-k}) when the full trajectories are observed. 
This family, considered by \cite{Collomb_1980}, combines the smoothness properties of the kernel function with 
the locality properties of the $k$-nearest neighbors distances.

Regarding regression when discretized curves are available, 
\cite{hart_1986} study the mean square consistency of the kernel estimator 
when the sample size as well as the grid size discretization go to infinity. More precisely, from
 independent reali\-zations of a random process  with continuous covariance structure, they estimate
the regression function, assuming its smoothness. Under the same assumptions, but using 
interpolation of the data, \cite{rice_1991},  in a mainly practical approach, propose a method to
estimate the regression function via smoothing splines (see  also \cite{hastie_1990}). More recently,
\cite{Cai_2011} establish minimax rates of convergence of estimators of the mean
based on discretized sampled data while \cite{Cai_2016} establish the minimax rates of convergence for the covariance operator when data are observed on a lattice (see also \cite{Hall_2006} for the problem of principal components analysis for longitudinal data). In this context it is natural to assess the relation
between the \textit{ideal} nonparametric regression estimator constructed with the entire
set of curves and the one computed with the discretized sample. In this direction, we are
interested in addressing the following question:
\begin{itemize}
\item  Under what conditions can the consistency (and rates of convergence) of the estimate 
computed with the discretized  trajectories be derived from the consistency of the estimate 
based on the full curves?
\end{itemize}

Clearly, the asymptotic  results for estimates computed with the discretized sample will not be a
direct consequence of those for the full model. However, we provide reasonable conditions
in order to still get the consistency and find rates of convergence of the estimator. In this context
we state the results for the well known kernel and $k$-nearest neighbor with kernel estimators. These results are a consequence of a more general result, which, besides discretization, also includes the cases of regularization via
 smoothing and basis representation.

This paper is organized as follows: In Section \ref{functional} we state the consistency of the
$k$-nearest neighbor with kernel estimator in the infinite dimensional setting (for the full model). \textcolor{black}{This result is not only interesting by itself but also, it will be used to prove consistency results when discretely sample data are available.} In Section \ref{general} we provide conditions for the consistency of the kernel and $k$-nearest neighbor with kernel estimators when we do not observe the whole trajectories but only a function of them (Theorems \ref{teo3} and \ref{teo4}). In Section \ref{particulares} the results for discretization, smoothing and basis representation are obtained as a consequence of Theorems \ref{teo3} and \ref{teo4}. Proofs are given in Appendices A and B.


\section{\textcolor{black}{Consistency results for fully observed curves}}\label{functional}

In this section we provide two $L^2$-consistency  results
for the full model, i.e., when ideally all trajectories are observed at every point
of the interval $[0,1]$. The first one corresponds to kernel estimates, and was obtained in 
\cite{forzani_fraiman_llop_2012}, while the second one for $k$-NN with kernel estimates is derived
in the present paper. \textcolor{black}{Both results will be used, in Section \ref{general}, to prove the consistency of that estimators when only discretely sampled curves in $[0,1]$ are observed.}

\

Let $(\Hi,d)$ be a separable metric space  and let $(\X_1,Y_1), \dots (\X_n, Y_n)$ be independent
identically distributed (i.i.d.) random elements in $\Hi \times \real$ with
the same law as the pair $(\X,Y)$ fulfilling the model:
\begin{equation}\label{modelo}
 Y = \eta(\X) + e,
\end{equation}
where the error $e$ satisfies  $\esp{e\vert\X}{e|\X}=0$ and $\var{e\vert\X}{e|\X} = \sigma^2 <
\infty$.
In this context, the regression function $E(Y|\X) = \eta(\X) $ can
be estimated by
\begin{equation}\label{estimador_general}
\etahat_n(\X) = \sum_{i=1}^n  W_{ni}(\X) Y_i,
\end{equation}
where the weights $W_{ni}(\X) = W_{ni}(\X,\X_1,\ldots,\X_n) \ge 0$ and
$\sum_{i=1}^n W_{ni}(\X) = 1$. In this paper, we first consider the weights corresponding to 
the family of kernel estimators given by
\begin{equation}\label{pesos_kernel}	
W_i(\X)=\frac{K\big(\frac{d(\X,\X_i)}{h_n(\X)}\big)}{\sum_{j=1}^n
K\big(\frac{d(\X,\X_j)}{h_n(\X)}\big)},
\end{equation}
where $K$ is a regular kernel, i.e., there are constants  $0< c_1 < c_2 < \infty$ such that $c_1
\ind_{[0,1]}(u) \le K(u) \le c_2 \ind_{[0,1]}(u)$. Here $0/0$ is assumed to be $0$. In this
general setting, \cite{forzani_fraiman_llop_2012} proved the following result.
\begin{proposition}[Theorem 5.1 in \cite{forzani_fraiman_llop_2012}] \label{consistencia_kernel}  Assume that
\begin{enumerate}
\item[K1)]  $K$ is a regular and Lipschitz kernel;
\item[F1)] $(\Hi,d)$  is a separable metric space;
\item[F2)] $\{(\X_i, Y_i)\}_{i\ge 1}$ are i.i.d. random elements with the same law
as the pair $(\X,Y) \in \Hi \times \real$ fulfilling model (\ref{modelo}) \textcolor{black}{with, for each $i=1,\ldots,n$, joint distribution $\mathbb{P}_{\X,\X_i}$};
\item[F3)] $\mu $ is a Borel probability measure of $\X$ and  $\eta \in
L^2(\Hi,\mu) = \{f: \Hi \to \real: \int_{\Hi} f^2 (z) d\mu(z) < \infty\}$ is a bounded function
which satisfies the \textit{Besicovitch condition}:
\begin{equation}\label{ec:besicovitch}
\lim_{\delta \to 0} \frac{1}{\muu{\mathcal{B}(\X, \delta)}} \int_{\mathcal{B}(\X, \delta)}
\abs{\eta(z)-\eta(\X)} \, d\mu(z) = 0,
\end{equation}
in probability, where $\mathcal{B}(\X, \delta)$  is the closed ball  of center $\X$ and
radius $\delta$ with respect to $d$.
\end{enumerate}
For any $\xf \in \sop{\mu}$ and any sequence $h_n(\xf) \to 0$ such that $\frac{n \mu(\mathcal{B}(\xf,h_n(\xf))}{\log n} \to
\infty$, the estimator given in (\ref{estimador_general}) with weights
given in (\ref{pesos_kernel}) satisfies
\[
\lim_{n \to \infty} \esp{}{(\etahat_n(\X) - \eta(\X))^2} =0.
\]
\end{proposition}
\begin{remark}\label{besicovitch}
The Besicovitch condition in \textit{F3} is a differentiation type condition which, as is well known, in finite dimensional spaces automatically holds for any integrable
function $\eta$. Unfortunately, it is no longer true in infinite dimensional spaces and it can be proved, for instance, that it is necessary in order to get the $L_1$-consistency of uniform kernel estimates (see Proposition 5.1 in \cite{forzani_fraiman_llop_2012}). However, it holds in a general setting if, for instance, the function $\eta$ is continuous. \textcolor{black}{For a deeper reading on this topic see \cite{cerou_and_guyader_2006} or \cite{forzani_fraiman_llop_2012}.}
\end{remark}
\begin{remark}\label{h-mayor}
Note that  for $\xf \in \sop{\mu}$ the consistency of this estimator holds for every sequence
$\tilde{h}_n (\xf)\to 0$ such that $\tilde{h}_n(\xf)\ge h_n(\xf)$, where $ h_n(\xf)$ is given in Proposition \ref{consistencia_kernel},  since if
$\tilde{h}_n(\xf)\ge h_n(\xf)$, then $\frac{n
\mu(\mathcal{B}(\xf,\tilde{h}_n(\xf))}{\log n} \ge \frac{n
\mu(\mathcal{B}(\xf,h_n(\xf))}{\log n} \to \infty$.
\end{remark}
The existence of a sequence verifying $\frac{n \mu(\mathcal{B}(\xf,h_n(\xf))}{\log n} \to \infty$ in Proposition \ref{consistencia_kernel} follows from the next lemma. 
\begin{lemma}[Lemma A.5 in \cite{forzani_fraiman_llop_2012}]\label{existehn}
For any $\xf \in \sop{\mu}$, there exists a sequence of positive real numbers $h_n(\xf)
\to 0$ such that $\frac{n \mu(\mathcal{B}(\xf,h_n(\xf))}{\log n} \to \infty$.
\end{lemma}

Let $H_n(\xf)$ be the distance from $\xf$ to its $k_n$-nearest neighbor among $\{\X_1,\ldots, \X_n\}$. Recall that the $k_n$-nearest neighbor of $\xf$ among $\{\X_1,\ldots, \X_n\}$ is the sample point $\X_i$ reaching the $k_n$th smallest distance to  $\xf$ in the sample. Then, when the bandwidth in (\ref{pesos_kernel}) is given by $H_n(\xf)$, we
obtain the family of $k_n$-nearest neighbor ($k$-NN) with kernel estimates. In order to get consistency for this family of estimators, we need to prove that $H_n(\xf) \to  0$, as stated in the following lemma. 
\begin{lemma} [Lemma A.4 in \cite{forzani_fraiman_llop_2012}] \label{lema:d(Xk,X)to0}
Let $\Hi$ be a separable metric space, $\mu$ a Borel probability measure,
and $\llave{\X_i}_{i=1}^n$ a random sample of $\X$. If $\xf \in
\sop{\mu}$ and $k_n$ is a sequence of positive real numbers such that $k_n \to  \infty$ and $k_n/n
\to  0$, then $H_n(\xf) \to  0$.
\end{lemma}
Although the distance from $\xf$ to its $k_n$-NN among $\{\X_1,\ldots, \X_n\}$
converges to zero, to prove first the consistency of this estimator, we cannot apply
directly Proposition \ref{consistencia_kernel} because we do not know that $H_n(\xf)$ satisfies $\frac{n \mu(\mathcal{B}(\xf,H_n(\xf))}{\log n} \to \infty$. However, as we will see in the next result, we can still prove the mean square consistency of this estimator under the same weak conditions as in Proposition \ref{consistencia_kernel}, whose proof can be found in the Appendix \ref{proofs}.
\begin{proposition} \label{consistencia_knn-k}
Assume K1, F1--F3 hold. Let $k_n$ be a sequence of positive real numbers such that $k_n\to \infty$, $k_n/n \to 0$ and let $H_n(\xf)$ be the distance from $\xf$ to its $k_n$-nearest neighbor among $\{\X_1,\ldots, \X_n\}$. Then, the estimator given by (\ref{estimador_general}) with weights given in
(\ref{pesos_kernel}) is mean square consistent for any sequence $h_n(\xf) \to 0$ such that
$h_n(\xf)\ge H_n(\xf)$, $\xf \in \sop{\mu}$.
\end{proposition}
%
\section{\textcolor{black}{Consistency results for discretely sampled curves}}
\label{general}

In this section we will assume that we are not able to observe the whole trajectories  $\X_i$ in $\Hi$ \textcolor{black}{given in $F2$}, but only a function of them. As we will see in Section \ref{particulares}, different choices of that function will correspond to discretizations, eigenfunction expansions, or smoothing. In this context, the weights of the estimator given in (\ref{pesos_kernel}) cannot
be computed because we have not a distance $d$ defined for the discretized sample curves (as a consequence, we do not have the validity of the Besicovitch condition (\ref{ec:besicovitch}) for the discretized 
data) or a bandwidth $h_n$. 

 We are interested in defining an estimator and proving its consistency in this setting. For that, let us consider the following assumptions:

{\it \begin{enumerate}
\item[H1)]\label{H1} \textcolor{black}{$(\Hi,d)$ is a separable (metric) Hilbert space} and $F: \Hi \to \textcolor{black}{\Hi}$ is a function such that, \textcolor{black}{for each $i=1,\ldots,n,$ $F(\X_i) = \X_i^p$;}
\item[H2)]\label{H2} $d_p : \Hi \times \Hi \to \real$ is a pseudometric \textcolor{black}{in $\Hi$} defined by $d_p(\X, \Y) =
d(\X^p,\Y^p)$ such that there exists a sequence $c_{n,p} \to 0$ as $n,p\to \infty$ satisfying, \textcolor{black}{for each $i=1,\ldots,n$,}	
\begin{equation}\label{desigualdad-metrica2}
n^2\esp{\X}{\mathbb{P}^2_{\X_i|\X}{\big(\abs{d(\X,\X_i)-d_p(\X, \X_i)} \ge c_{n,p} \Big|\X \in
\sop{\mu}}\big)} \to 0.
\end{equation}
\end{enumerate}
Here, $\mathbb{P}^2_{\Y|\X}(\cdot)$ means the square of $\mathbb{P}_{\Y|\X}(\cdot)$.}

\begin{remark}
Observe that in $H1$ neither $\Hi$ nor $F$ change with the sample. This implies that in this case, the functional data falls into the category of sparsely and regularly sampled data.
\end{remark}

The estimator of $\eta$ based on $\{(\X_i^p,Y_i)\}_{i=1}^n$ will be defined as in (\ref{estimador_general}) and (\ref{pesos_kernel}) but with the pseudometric $d_p$ instead of the metric $d$. More precisely, for $h_{n,p}(\X)>0$, we define
\begin{equation}\label{estimador-nuevo}
\etahat_{n,p}(\X) = \frac{\sum_{i=1}^n K\pa{\frac{d_p(\X,\X_i)}{h_{n,p}(\X)}}Y_i}{\sum_{j=1}^n K\pa{\frac{d_p(\X,\X_j)}{h_{n,p}(\X)}}}.
\end{equation}
For this estimator, we state the following two asymptotic results.
\begin{thm}\label{teo3}
Assume K1, F2, F3, H1 and H2 hold.
\begin{enumerate}[(a)] 
\item\label{kernel2} (Kernel estimator) For any $\xf \in \sop{\mu}$, let $h^*_n(\xf)\to 0$ be a sequence of positive real numbers such that $\frac{n \mu(\mathcal{B}(\xf,h^*_n(\xf))}{\log n} \to \infty$. Then, for  $c_{n,p}$ given in H2 and
 $h_{n,p}(\xf) \to 0$ such that there exists a sequence $h_n(\xf)\to 0$, $h_n(\xf) \ge
h^*_n(\xf)$ satisfying:
\begin{enumerate}
\item[(H3.1)]\label{H31} $\esp{\X}{c^2_{n,p}/h^2_n(\X)} \to 0$ as $n,p\to \infty$;
\item[(H3.2)]\label{H32} $c_{n,p} \le h_{n,p}(\xf) - h_n(\xf) \le C_2 c_{n,p}$ for $C_2 \ge  1$;
\end{enumerate}
we have
\begin{equation}\label{dos}
\lim_{n,p \to \infty} \esp{}{(\etahat_{n,p}(\X) - \eta(\X))^2} =0.
\end{equation}
\item ($k_n$-NN with kernel estimator) Let $c_{n,p}$ given in H2 and $H_n(\xf)$ the distance from $\xf$ to its $k_n$-nearest neighbor  among $\{\X_1,\ldots, \X_n\}$. For any $\xf \in \sop{\mu}$, let $h_{n,p}(\xf) \to 0$ be such that there exists a sequence $h_n(\xf)\to 0$, $h_n(\xf) \ge H_n(\xf)$ satisfying assumptions (\textit{H3.1}) and
 (\textit{H3.2}). Then, for $k_n\to \infty$ and  $k_n/n \to  0$ we have (\ref{dos}).

\end{enumerate}
\end{thm}
\begin{remark}\label{existehnp}
Observe that the sequence $h^*_n(\xf)$ in Theorem \ref{teo3} always exists by
Lemma \ref{existehn}. In
addition, under \textit{H2}, it is always possible to choose a sequence $h_{n,p}(\xf) \to 0$ 
{fulfilling the conditions} in Theorem \ref{teo3}. Indeed, taking $h_n(\xf) = \max\{h_n^*(\xf), \sqrt{c_{n,p}}\}$ and $h_{n,p}(\xf) = h_n(\xf) + C c_{n,p}$,  with $C \ge 1$, we have that $h_n(\xf) \to 0$, 
$h_{n,p}(\xf) \to 0$, $h_n(\xf) \ge h_n^*(\xf)$,  (\textit{H3.1}) holds since $h_n(\xf) \ge \sqrt{c_{n,p}}$ and (\textit{H3.2}) holds by definition of $h_{n,p}(\xf)$.  The same happens if instead of taking $h^*_n(\xf)$ we take $H_n(\xf)$.
\end{remark}
\begin{thm}\label{teo4}
Under the assumptions of Theorem \ref{teo3}, let $\gamma_n \to \infty$ as $n\to
\infty$ be such that, as $n,p\to \infty$,
\begin{enumerate}[(a)]
 \item $\esp{\X}{\gamma_n \left(\frac{c_{n,p}}{h_n(\X)}\right)^2} \to 0$;
 \item $\gamma_n n^2\esp{\X}{\mathbb{P}^2_{\X_i|\X}{\big(\abs{d(\X,\X_i)-d_p(\X, \X_i)} \ge c_{n,p}
\Big|\X \in \sop{\mu}}\big)} \to 0$, for each $i=1,\ldots,n$.
\end{enumerate}
Then
\[
\lim_{n \to \infty}\esp{}{\gamma_n(\etahat_n(\X) - \eta(\X))^2} = 0,
\]
implies 
\[ 
\lim_{n,p \to \infty}  \esp{}{\gamma_n (\etahat_{n,p}(\X) - \eta(\X))^2} = 0.
\]
\end{thm}
\section{Particular cases}\label{particulares}

In this section we provide definitions of $\Hi_p$ and $d_p$ for discretization, smoothing, and eigenfunction expansions, which satisfy conditions  \textit{H1} and \textit{H2}. Then, for any sequence $h_{n,p}(\xf)\to 0$ satisfying (\textit{H3.1}) and (\textit{H3.2}) in Theorem \ref{teo3}, we get the consistency of $\hat\eta_{n,p}$ as a consequence of the consistency results for $\hat\eta_n$ in the full model.

\

Consider the case where the elements of the dataset are curves in $L^2([0,1])$  that are only observed at a discrete set of points in the interval $[0,1]$.  More precisely, let us assume that $\{\X_i\}_{i=1}^n$ are observed only at some points: $
(\X_i(t_1), \ldots, \X_i(t_{p+1}))$
where  $0=t_1 <t_2 \le \ldots < t_{p+1} = 1$, which for simplicity we will assume
are equally spaced, i.e., $\Delta t = t_{i+1} - t_i = 1/p$. In this case, we will need to require the trajectories to satisfy some regularity condition. More precisely, we will assume that $\X$ is a random element of $\Hi \doteq H^1([0,1])$,
the Sobolev space defined as
\[
H^1([0,1]) = \{f:[0,1] \to \real: f \text{ and } Df \in L^2([0,1])\},
\]
where $Df$ is the weak derivative of $f$, i.e., $Df$ is a function in $L^2([0,1])$
which satisfies
\[
\int_0^1 f(t) D\phi(t)\, dt = - \int_0^1 Df(t) \phi(t)\, dt,  \hspace{1cm} \forall \, \phi
\in C_0^{\infty}.
\]
In this space, the norm is defined by
\[
\norm{f}_{H^1([0,1])} = \norm{f}_{L^2([0,1])} + \norm{Df}_{L^2([0,1])}.
\]
In this setting, we will prove consistency for the pseudometrics $d_p$ given below.


\subsection{\rm \textit{Discretization}}\label{discretizacion}

Consider the pseudometric
\[d_p(\X,\X_1) = d(\X^p,\X_1^p)= \pa{\frac{1}{p} \sum_{j=1}^p\abs{\X(t_j) - \X_1(t_j)}^2}^{1/2} ,\]
where $\X^p(t) = F(\X)(t) = \sum_{j=1}^p \phi_j(t) \X(t_j)$ with $\phi_j(t) = \ind_{[t_j,t_{j+1})}(t)$. In this case, consistency will hold for any sequence $c_{n,p}\to 0$ as $n,p \to \infty$ such that $n^2\prob{\X,\X_1}{\norm{\X}_{\Hi} +  \norm{\X_1}_{\Hi} \ge p c_{n,p}} \to 0$. 
 
%

%
\subsection{\rm \textit{Kernel Smoothing}}\label{kernel}

Let us consider now the pseudometric
\[d_p(\X,\X_1) = d(\X^p,\X_1^p) = \pa{\int_0^1 \abs{\X^p(t)-\X_1^p(t)}^2 \, dt}^{1/2},\]
where $\X^p(t) = F(\X)(t) = \sum_{j=1}^p \phi_j(t) \X(t_j)$ with $\phi_j(t) = \frac{K(|t-t_j|/h)}{\sum_{i=1}^p K(|t-t_i|/h)}$ and $K$ is a regular kernel supported in $[0,1]$. In this case, consistency will be true for any sequence $c_{n,p}\to 0$ as $n,p \to \infty$ satisfying $n^2\prob{\X,\X_1}{\norm{\X}_{\Hi} +  \norm{\X_1}_{\Hi} \ge p c_{n,p}} \to 0$.

Let us note that if $\mathbb E_{\X}(\norm{\X}^2_{\Hi})<\infty$, the consistency for the cases given in Sections \ref{discretizacion} and \ref{kernel} will hold for any sequence $c_{n,p} $ such that $ \frac{ n}{p c_{n,p}} \rightarrow 0$.

\subsection{\rm \textit{Eigenfunction expansions}}\label{eigen}

 Let $\X, \X_1$ be i.d. random elements on $\Hi=L^2[0,1]$. Let $v_1, v_2,  \dots$ be the orthonormal
eigenfunctions of the covariance operator $\esp{\X}{\X(t)\X(s)}$ (without loss of generality we 
have
assumed that $\esp{}{\X(t)} = 0$) associated with the eigenvalues $\lambda_1 \ge \lambda_2\ge
\ldots$ such that
\[
\esp{\X}{\X(t)\X(s)} = \sum_{k=1}^\infty \lambda_k v_k(t) v_k(s).
\]
If $\esp{}{\int \X^2(s) \, ds} < \infty$ is finite, using the Karhunen--Lo\`eve representation,
we can write $\X$ as
\begin{equation}\label{X}
 \X(t) = \sum_{k=1}^\infty \pa{\int \X(s) v_k(s) \, ds} v_k(t) \doteq \sum_{k=1}^\infty \xi_k
v_k(t),
\end{equation}
with $\esp{}{\xi_k} = 0$, $\esp{}{\xi_k \xi_j} = 0$ (i.e., $\xi_1, \xi_2, \ldots$
uncorrelated) and $\var{}{\xi_k} = \esp{}{\xi_k^2} = \lambda_k =  \esp{}{\pa{\int \X(s) v_k(s) \, ds}^2} $. The classical $L^2$-norm in $\Hi$
can be written as
\begin{equation}\label{dL2}
d(\X,\X_1) =  \sqrt{\sum_{k=1}^{\infty} \pa{\int (\X(t) - \X_1(t)) v_k(t) \,dt}^2}.
\end{equation}
If we consider the truncated expansion of $\X$ as given in \cite{Ferraty_and_vieu_2006},
\begin{equation}\label{X-ferraty}
\X^p(t) = \sum_{k=1}^p \pa{\int \X(s) v_k(s) \, ds} v_k(t),
\end{equation}
we can define the parametrized class of seminorms from the classical $L^2$-norm given by
\[
\norm{\X}_p = \sqrt{\int (\X^p(t))^2 \, dt} = \sqrt{\sum_{k=1}^p \pa{\int \X(t) v_k(t) \,
dt}^2},
\]
which leads to the pseudometric
\begin{equation}\label{d-ferraty}
d_p(\X, \X_1) = d(\X^p, \X_1^p) = \sqrt{\sum_{k=1}^p \pa{\int (\X(t) - \X_1(t)) v_k(t) \,dt}^2}.
\end{equation}
In this case, the consistency will hold for any sequence $c_{n,p}\to 0$ such that $\frac{n^2}{c^2_{n,p}} \sum_{k=p+1}^{\infty}\lambda_k \to 0$ as $n,p \to \infty$.


\appendix

\section{Proofs of auxiliary results}\label{auxiliares}

To prove the consistency of the examples given in sections \ref{discretizacion} and \ref{kernel} we need the following result.
\begin{proposition}\label{lema2}
	Let $\X^p(t) = \sum_{j=1}^p \phi_j(t) \X(t_j)$ with $\phi_j$ satisfying:
	\begin{enumerate}[(a)]
		\item \label{1} for each $t \in [0,1]$, $\sum_{j=1}^p \phi_j(t) =1$;
		\item \label{3} for each $t \in [0,1]$, $\sum_{j=i}^p \phi_j^2(t) \le C_3$ for some constant $C_3$;
		\item \label{2} $\sop{\phi_j} \subset [t_{(j-m)}, t_{(j+m)}]$ with $m$ independent of $p$.
	\end{enumerate}
	If $c_{n,p}\to 0$ as $n,p \to \infty$ is such that $n^2\prob{\X,\X_1}{\norm{\X}_{\Hi} +  \norm{\X_1}_{\Hi} \ge p c_{n,p}} \to 0$, \textit{H2} is fulfilled. 
\end{proposition}

\begin{proof}[Proof of Proposition \ref{lema2}]
	Using the Fundamental Theorem of Calculus (FTC) (see Theorem 8.2 in \cite{Brezis_2010}) for $H^1([0,1])$, we get

	\begin{align*}
	d^2(\X^p,\X) &= \int_0^1 \abs{\sum_{j=1}^p \X(t_j)\phi_j(t) - \X(t)}^2 \, dt \\
	&= \int_0^1 \abs{\sum_{j=1}^p (\X(t_j) - \X(t)) \phi_j(t)}^2 \, dt  &\pa{\text{by
			(\ref{1})}}\\ &= \int_0^1 \abs{\sum_{j=1}^p \pa{\int_{t_j}^t D\X(s)
			\, ds} \phi_j(t)}^2 \, dt &\pa{\text{from FTC}} \\&\le \int_0^1 \hspace{-0.2cm} \pa{\sum_{j=1}^p
		\pa{\int_{t_j}^t D\X(s) \, ds}^2 \hspace{-0.1cm}  \ind_{\{\sop{\phi_j}\}}(t)}\hspace{-0.1cm} \pa{ \sum_{j=1}^p \phi_j^2(t) } \, dt
	&\pa{\text{by C-S Ineq.}}
	\\&\lesssim \int_0^1 \sum_{j=1}^p  \pa{\int_{t_j}^t D\X(s) \, ds}^2 \ind_{\{\sop{\phi_j}\}}(t) \, dt &\pa{\text{by (\ref{3})}} \\&\lesssim \int_0^1\sum_{j=1}^p
	\pa{\int_{t_j}^t (D\X(s))^2 \, ds} |t-t_j| \ind_{\{\sop{\phi_j}\}}(t) \, dt & \pa{\text{by C-S
			Ineq.}}
	\\&= \sum_{i=1}^p \int_{t_i}^{t_{i+1}} \sum_{\stackrel{j=1}{\stackrel{j:|j-i|\le m}{}}}^p
	\pa{\int_{t_j}^t (D\X(s))^2 \, ds} |t-t_j| \, dt &\pa{\text{by (\ref{2})}} \\&\lesssim \sum_{i=1}^p \sum_{\stackrel{j=1}{\stackrel{j:|j-i|\le m}{}}}^p
	\int_{t_{i-m}}^{t_{i+m}} (D\X(s))^2 \pa{\int_{t_j}^{t_{j+1}} |t-t_j| \, dt} \, ds \\&\lesssim  \frac{m}{p^2}\sum_{i=1}^p \sum_{\stackrel{j=1}{\stackrel{j:|j-i|\le m}{}}}^p
	\int_{t_{i-m}}^{t_{i+m}} (D\X(s))^2 \, ds \\&\lesssim   \frac{m^ 2}{p^2}\sum_{i=1}^p
	\int_{t_{i-m}}^{t_{i+m}} (D\X(s))^2 \, ds \\&=  \frac{m^ 2}{p^2}\int_0^1\sum_{i=1}^p
	\ind_{[t_{i-m}, t_{i+m}]}(s) (D\X(s))^2 \, ds \lesssim  \frac{1}{p^2} \norm{\X}_{\Hi}^2,
	\end{align*}
	from where we get $d(\X^p,\X)\lesssim  \frac{1}{p}  \norm{\X}_{\Hi}$. Analogously we can prove that $d(\X_1^p,\X_1)\lesssim  \frac{1}{p}  \norm{\X_1}_{\Hi}$. By triangular inequality,
	\begin{align*}
	&\nonumber \hspace{-2cm} n^2\esp{\X}{\mathbb{P}^2_{\X_1|\X}\big(\abs{d(\X,\X_1)-d_p(\X, \X_1)} \ge c_{n,p} \Big|\X
		\in\sop{\mu}\big)} \\&\le n^2\prob{\X,\X_1}{\norm{\X}_{\Hi} +  \norm{\X_1}_{\Hi} \ge p c_{n,p}},
	\end{align*}
	and therefore, for any $c_{n,p}\to 0$ such that $n^2\prob{\X,\X_1}{\norm{\X}_{\Hi} +  \norm{\X_1}_{\Hi} \ge p c_{n,p}} \to 0$ \textit{H2} is fulfilled.
\end{proof}

\subsubsection{Consistency for the example in Section \ref{discretizacion}}
Since the functions $\phi_j(t) = \ind_{[t_j,t_{j+1})}(t)$ satisfy trivially conditions (\ref{1})--(\ref{2}) of Proposition \ref{lema2}, \textit{H2} is fulfilled and therefore, for any sequence $h_{n,p}(\xf)\to 0$ satisfying (\textit{H3.1}) and (\textit{H3.2}) in Theorem \ref{teo3}, we get the consistency of $\hat\eta_{n,p}$.

\subsubsection{Consistency for the example in Section \ref{kernel}}

Observe that $\phi_j(t) = \frac{K(|t-t_j|/h)}{\sum_{i=1}^p K(|t-t_i|/h)}$ satisfies conditions 
(\ref{1})--(\ref{2}) in Proposition \ref{lema2}:
\begin{enumerate}[(a)]
	\item  for each $t \in [0,1]$, $\sum_{j=1}^p \phi_j(t) = \sum_{j=1}^p
	\frac{K(|t-t_j|/h)}{\sum_{i=1}^p K(|t-t_i|/h)} = 1$;
	\item since $K$ is nonnegative and $\frac{K(|t-t_j|/h)}{\sum_{i=1}^p K(|t-t_i|/h)}\le 1$, for each $t \in [0,1]$, there exists $C_3=1$ such that,
	$$\sum_{j=1}^p \phi_j^2(t) =
	\sum_{j=1}^p \pa{\frac{K(|t-t_j|/h)}{\sum_{i=1}^p K(|t-t_i|/h)}}^2 \le \sum_{j=1}^p \frac{K(|t-t_j|/h)}{\sum_{i=1}^p K(|t-t_i|/h)} = 1;$$
	\item $\sop{\phi_j} = \sop{K(|t-t_j|/h)} = [t_j-h, t_j + h]$, which implies that, for $h \le m/p$,
	$\sop{\phi_j}  \subset [t_{(j-m)}, t_{(j+m)}]$.
\end{enumerate}
This implies that \textit{H2} is fulfilled then, for any sequence $h_{n,p}(\xf)\to 0$ satisfying (\textit{H3.1}) and (\textit{H3.2}) in Theorem \ref{teo3}, we get the consistency of $\hat\eta_{n,p}$.  

\subsubsection{Consistency for the example in Section \ref{eigen}}

Let us consider the truncated expansion of $\X$, $\X^p(t)$, given by (\ref{X-ferraty}) and
the pseudo-metric $d_p(\X,\X_1) = d(\X^p,\X_1^p)$ given by (\ref{d-ferraty}). In order to prove
\textit{H2}, let us consider $c_{n,p}$ such that $\frac{n^2}{c^2_{n,p}} \sum_{k=p+1}^{\infty}
\lambda_k \to 0$. Using Chebyshev's Inequality in (\ref{desigualdad-metrica2}) followed by Cauchy Schwartz, we get
\begin{align}\label{dd}
&\nonumber \hspace{-2cm}n^2\esp{\X}{\mathbb{P}^2_{\X_1|\X}\big(\abs{d(\X,\X_1)-d_p(\X, \X_1)} \ge c_{n,p} \Big|\X
	\in\sop{\mu}\big)} \\&\le \frac{n^2}{c^2_{n,p}}\mathbb{E}_{\X,\X_1}\big(\pa{d(\X,\X_1)-d_p(\X, \X_1)}^2\big).
\end{align}
Now, since $d(\X,\X_1) \ge d_p(\X,\X_1)$ we have that $0\le d(\X,\X_1) - d_p(\X,\X_1) = d(\X,\X_1) - d(\X^p,\X_1^p)$ and, by triangular inequality $d(\X,\X_1) \le  d(\X,\X^p) + d(\X^p,\X_1^p) + d(\X_1^p,\X_1)$ which implies that,
\begin{equation}\label{triangular1}
0 \le d(\X,\X_1) - d_p(\X,\X_1) \le d(\X,\X^p) + d(\X_1^p,\X_1),
\end{equation}
and taking squares,
\[
0 \le (d(\X,\X_1) - d_p(\X,\X_1))^2 \le (d(\X,\X^p) + d(\X_1^p,\X_1))^2 \le  2\pa{d^2(\X,\X^p) + d^2(\X_1^p,\X_1)}.
\]
As a consequence, to proof this proposition it will sufficient to bound $\esp{\X}{d^2(\X,\X^p)}$ (equivalently, $\esp{\X_1}{d^2(\X_1,\X_1^p)}$). Since $v_k$ are orthonormal, 
\begin{align*}
d^2(\X, \X^p) &= \int \left( \X(s) - \sum_{k=1}^p \pa{\int \X(t) v_k(t) \, dt} v_k(s)\right)^2 \,
ds \\ &= \sum_{k=p+1}^{\infty} \pa{\int \X(t) v_k(t) \, dt}^2\hspace{-0.1cm}.
\end{align*}
Then we have,
\begin{align*}
\esp{\X}{d^2(\X,\X^p)} &= \esp{\X}{\sum_{k=p+1}^{\infty} \pa{\int \X(t) v_k(t) \, dt}^2}
\\ &= \sum_{k=p+1}^{\infty}
\lambda_k. & \pa{\text{from } (\ref{X}) }
\end{align*}
Analogously we can prove that $\esp{\X_1}{d^2(\X_1,\X_1^p)} =\sum_{k=p+1}^{\infty}
\lambda_k$. Therefore, in (\ref{dd}) we get
\begin{align*}
& n^2\esp{\X}{\mathbb{P}^2_{\X_1|\X}\big(\abs{d(\X,\X_1)-d_p(\X, \X_1)} \ge c_{n,p} \Big|\X
	\in\sop{\mu}\big)}  \\&\hspace{2cm}\lesssim \frac{n^ 2}{c^2_{n,p}} \sum_{k=p+1}^{\infty} \lambda_k \to 0.
\end{align*}
This implies that \textit{H2} is fulfilled then, for any sequence $h_{n,p}(\xf)\to 0$ satisfying (\textit{H3.1}) and (\textit{H3.2}) in Theorem \ref{teo3}, we get the consistency of $\hat\eta_{n,p}$.  

\section{Proof of Proposition \ref{consistencia_knn-k} and Theorems \ref{teo3} and \ref{teo4}}\label{proofs}
Here (and hereafter) { we will use the notation} $f \lesssim g$ when there exists a
constant $C>0$ such that $f \le C g$ and $f \approx g$ if there exists a constant $C>0$ such that $f
= C g$. 
To prove Proposition \ref{consistencia_knn-k} we need some preliminary results whose proofs can be
found in \cite{forzani_fraiman_llop_2012}. 
\begin{thm}[Theorem 3.4]\label{stone_version}
	If $\eta \in L^2(\Hi,\mu)$ and $\etahat_n$ is the estimator given in (\ref{estimador_general}) with weights $W_n (\X) = \llave{W_{ni}(\X)}_{i=1}^n$ satisfying the following conditions:
	\begin{itemize}
		\item[(i)] \label{i}There is a sequence of nonnegative random va\-ria\-bles $a_n(\X) \to
		0$ a.s. such that
		\[
		\lim_{n\to \infty} \esp{}{\sum_{i=1}^n W_{ni}(\X) \ind_{\llave{d(\X,\X_i)>a_n(\X)}}} =0;
		\]
		\item[(ii)]\label{ii}
		\[
		\lim_{n\to \infty} \esp{}{\max_{1\le i \le n} W_{ni}(\X)} =0;
		\]
		\item[(iii)]\label{iii} for all $\epsilon >0$ there exists $\delta >0 $ such that for any
		$\eta^*$
		bounded and continuous function fulfilling \mbox{$\mathbb{E}_{\X}((\eta(\X) -
			\eta^*(\X))^2) < \delta$} we have that
		\[
		\esp{}{\sum_{i=1}^n W_{ni}(\X)(\eta^*(\X_i)-\eta(\X_i))^2} <
		\epsilon,
		\]
	\end{itemize}
	then $\etahat_n$ is mean square consistent.
\end{thm}
\begin{corollary}[Corollary 3.3]\label{cor:2stone}
	Let $U_n$ be a sequence of pro\-ba\-bility weights satisfying conditions (i), (ii) and
	(iii) of Theorem \ref{stone_version}. If $W_n$ is a sequence of weights such that
	$\sum_{i=1}^n W_{ni}(\X) = 1$ and, for each \mbox{$n\ge 1$},
	$\abs{W_n}\le M U_n$ for some constant $M\ge 1$, then the estimator  given in (\ref{estimador_general}) with weights $W_n (\X)$ is mean square consistent.
\end{corollary}
\begin{lemma}[Lemma A.1]\label{prob_sop}
	Let $\Hi$ be a separable metric space. If $A = \sop{\mu} = \{ x \in
	\Hi :\muu{\bola{x}{\epsilon}}>0, \forall \, \epsilon>0)\}$ then $\muu{A} = 1$.
\end{lemma}

\begin{proof}[Proof of Proposition \ref{consistencia_knn-k}]
	Let $\xf \in \sop{\mu}$ be fixed. Let us observe that, since $K$ is regular, there {exist} constants
	\mbox{$0<c_1 < c_2 < \infty$} such that, for each $i$,
	\begin{equation}\label{def_Uni}
	W_{ni}(\xf) = \frac{K\pa{\frac{d(\X_i, \xf)}{h_n(\xf)}}}{\sum_{j=1}^n
		K\pa{\frac{d(\X_j, \xf)}{h_n(\xf)}}} \le \frac{c_2}{c_1}  \frac{\ind_{\llave{ d(\X_i,
				\xf)\le h_n(\xf)}}}{\sum_{j=1}^n \ind_{\llave{d(\X_j, \xf) \le h_n(\xf)}}} \doteq
	\frac{c_2}{c_1}U_{ni}(\xf).
	\end{equation}
	Let $h_n(\xf)\to 0$ such that $h_n(\xf) \ge H_n(\xf)$ ($H_n(\xf)\to 0$ by Lemma
	\ref{lema:d(Xk,X)to0}, for $\xf \in \sop{\mu}$). From (\ref{def_Uni}) and
	Corollary \ref{cor:2stone}, it suffices to prove that the weights $U_{ni}$ satisfy
	conditions (i), (ii) and (iii) of Theorem \ref{stone_version}. To prove (i) let us take
	$a_n(\xf) = h_n^{1/2}(\xf) \to 0$. Then, by Lemma \ref{prob_sop},
	\begin{align*} \label{ec:4cap4}
	\esp{}{\sum_{i=1}^n U_{ni}(\X) \ind_{\llave{d(\X_i, \X)>h_n(\X)^{1/2}}}}
	\\&\hspace{-5.7cm}=\esp{\X}{\esp{\Dn|\X}{\ind_{\{\X \in
				\sop{\mu}\}}\sum_{i=1}^n U_{ni}(\X) \ind_{\llave{d(\X_i, \X)>h_n(\X)^{1/2}}}\Big| \X \in \sop{\mu}}}.
	\end{align*}
	Given $\epsilon>0$, let $\xf \in \sop{\mu}$ be fixed. Since
	\mbox{$h_n(\xf) \to 0$}, there exists $N_1 = N_1(\xf)$ such that if
	$n\ge N_1$, $\ind_{\llave{h_n(\xf)^{1/2} < d(\xf_i, \xf) \le h_n(\xf)}} =
	0$ for all $i$ and consequently,
	\[
	\esp{\Dn}{\frac{1}{\sum_{j=1}^n \ind_{\llave{d(\xf_j, \xf) \le h_n(\xf)}}}\sum_{i=1}^n
		\ind_{\llave{h_n(\xf)^{1/2}< d(\xf_i, \xf) \le h_n(\xf)}}} < \epsilon.
	\]
	In addition, $\frac{\sum_{i=1}^n\ind_{\llave{h_n(\xf)^{1/2}< d(\xf_i, \xf) \le
				h_n(\xf)}}}{\sum_{j=1}^n \ind_{\llave{d(\xf_j, \xf) \le h_n(\xf)}}}  \le 1$ from what
	follows that,
	\[
	\esp{\Dn}{\frac{1}{\sum_{j=1}^n \ind_{\llave{d(\xf_j, \xf) \le h_n(\xf)}}}\sum_{i=1}^n
		\ind_{\llave{h_n(\xf)^{1/2}< d(\xf_i, \xf) \le h_n(\xf)}}} \le 1.
	\]
	Therefore, by the dominated convergence theorem we have that condition (i) is satisfied. Now, since 
	$h_n(\xf)\ge H_n(\xf)$, 
	$$\sum_{j=1}^n \ind_{\llave{d(\X_j, \xf) \le h_n(\xf)}} \ge \sum_{j=1}^n \ind_{\llave{d(\X_j, \xf) \le
			H_n(\xf)}} = k_n \rightarrow \infty.$$
	Therefore,
	\[
	\max_{1\le i \le n} U_{ni}(\xf) \le \max_{1\le i \le n} \frac{1}{\sum_{j=1}^n
		\ind_{\llave{d(\X_j, \xf) \le h_n(\xf)}}} \le \frac{1}{k_n} \to 0,
	\]
	from what {we derive (ii)  using} the  dominated convergence theorem. It remains to verify
	that condition (iii) holds. Since $\eta \in L^2(\Hi, \mu)$ which is separable and complete, there exists $\eta^*$ continuous and bounded such that, for all $\delta>0$, $\mathbb{E}_{\X}((\eta(\X)- \eta^*(\X))^2) < \delta$.  Then,
	\begin{align*}
	&\esp{}{\sum_{i=1}^n U_{ni}(\X)(\eta^*(\X_i)-\eta(\X_i))^2}
	\\&= \esp{\X}{\esp{\Dn|\X}{\ind_{\{\X \in \sop{\mu}\}}\sum_{i=1}^n
			U_{ni}(\X) (\eta^*(\X_i)-\eta(\X_i))^2 \vert \X \in\sop{\mu} }}.
	\end{align*}
	Let $\xf \in \sop{\mu}$ be fixed. From \cite{forzani_fraiman_llop_2012}, Lemma A.7, for any nonnegative bounded measurable function $f$, we have
	\[
	\esp{\Dn}{\sum_{i=1}^n U_{ni}(\xf)f(\X_i)}\le 12 \frac{1}{\muu{\mathcal{B}(\xf, h_n(\xf))}}
	\int_{\mathcal{B}(\xf, h_n(\xf))} f(y)  \, d\mu(y).
	\]
	Then, applying the inequality to $f(\X_i) = (\eta^*(\X_i)-\eta(\X_i))^2$ we get
	\begin{align*}
	&\esp{\Dn}{\sum_{i=1}^n U_{ni}(\xf)
		(\eta^*(\X_i)-\eta(\X_i))^2}  \nonumber \\&\hspace{0.5cm}\lesssim
	\frac{1}{\muu{\mathcal{B}(\xf,h_n(\xf))}} \int_{\mathcal{B}(\xf,h_n(\xf))}
	(\eta^*(y)-\eta(y))^2  \, d\mu(y) \\&\hspace{0.5cm}\le
	\frac{1}{\muu{\mathcal{B}(\xf,h_n(\xf))}}  \int_{\mathcal{B}(\xf,h_n(\xf))}
	(\eta^*(y)-\eta^*(\xf))^2  \, d\mu(y)  \\&\hspace{1cm}+
	\frac{1}{\muu{\mathcal{B}(\xf,h_n(\xf))}}  \int_{\mathcal{B}(\xf,h_n(\xf))}
	(\eta^*(\xf)-\eta(\xf))^2  \, d\mu(y)  \\&\hspace{1cm}+
	\frac{1}{\muu{\mathcal{B}(\xf,h_n(\xf))}}  \int_{\mathcal{B}(\xf,h_n(\xf))}
	(\eta(\xf)-\eta(y))^2 \, d\mu(y)\\&\hspace{0.5cm} \\&\doteq  f_{1,n}(\xf) +
	f_{2,n}(\xf) + f_{3,n}(\xf)).
	\end{align*}
	This part will be complete if we show that the expectation with respect to $\X$ of these
	three functions converges to zero. For this, let $\epsilon>0$ and $\delta\le \epsilon$.
	Since $\eta^*$ is continuous, there exists $r = r(\xf,\epsilon)>0$ such that if
	$d(\xf,y)<
	r$ then $\vert \eta^*(\xf) - \eta^*(y) \vert < \epsilon$. On the other hand, since $h_n(\xf) \to 0$, for that $r(\xf,\epsilon)>0$,  there exists $N_2 = N_2(\xf,r(\xf,\epsilon))$ such that if $n \ge N_2$,
	$h_n(\xf) < r$. Then, $f_{1,n}(\xf) = \frac{1}{\muu{\mathcal{B}(\xf,h_n(\xf))}}
	\int_{\mathcal{B}(\xf,h_n(\xf))} (\eta^*(y)-\eta^*(x))^2  \, d\mu(y) < \epsilon$ for $n \ge N_2$
	and in addition it is bounded so, by the dominated convergence theorem we have that
	\[
	\mathbb E_{\X}(f_{1,n}(\X))  \to 0.
	\]
	{For the second term, since $\delta \le \epsilon$, we have that}
	\[
	\mathbb E_{\X}(f_{2,n}(\X)) = \mathbb{E}_{\X}((\eta(\X)- \eta^*(\X))^2) < \epsilon.
	\]
	Finally, since $\eta $ is bounded,
	\[
	\mathbb E_{\X}(f_{3,n}(\X))\lesssim \esp{\X}{ \frac{1}{\muu{\mathcal{B}(\X,h_n(\X))}}
		\int_{\mathcal{B}(\X,h_n(\X))} \abs{\eta(\X)-\eta(y)} \, d\mu(y)},
	\]
	which converge to zero if the bounded random variables
	$$ \frac{1}{\muu{\mathcal{B}(\X,h_n(\X))}}  \int_{\mathcal{B}(\X,h_n(\X))}
	\abs{\eta(\X)-\eta(y)} \, d\mu(y)$$ converge to zero in probability. To see this, let $\lambda>0$ be
	fixed. For every $\delta_0>0$,
	\begin{align*}
	&\prob{\X}{\frac{1}{\muu{\mathcal{B}(\X,h_n(\X))}}
		\int_{\mathcal{B}(\X,h_n(\X))} \abs{\eta(\X)-\eta(y)} \, d\mu(y)>\lambda}  \\ &\le
	\prob{\X}{h_n(\X) > \delta_0} + \sup_{\delta \le \delta_0}
	\prob{\X}{\frac{1}{\muu{\mathcal{B}(\X,\delta)}}
		\int_{\mathcal{B}(\X,\delta)} \abs{\eta(\X)-\eta(y)} \, d\mu(y) >\lambda}.
	\end{align*}
	Since $h_n(\X) \to 0$ a.s. the first term converges to zero while the second term does thanks to the truth of the Besicovitch condition (\ref{ec:besicovitch}).
\end{proof}

\begin{proof}[Proof of Theorem \ref{teo3}]
	
	\
	
	\textit{Proof of (a):} Let us define $\Dn = \{\X_1,\ldots,\X_n\}$ and $\Cn =
	\{Y_1,\ldots,Y_n\}$. In order to prove the mean square consistency, we consider
	\[
	\esp{}{(\etahat_{n,p}(\X) - \eta(\X))^2)}  = \esp{\X}{ \esp{\Dn,\Cn|\X}{(\etahat_{n,p}(\X) -
			\eta(\X))^2)\big | \X}}.
	\]
	Let $\xf \in \sop{\mu}$ be fixed.  {To simplify the  notation, we set $
		\esp{}{\cdot} = \esp{\Dn,\Cn|\X}{\cdot}$. } Then, for a particular $h_n (\xf) \ge h^*_n(\xf)$ to be
	defined later, let us define the {\it theoretical quantities}
	\[
	K\pa{\frac{d(\xf,\X_i)}{h_n(\xf)}} \doteq K_i(\xf) \doteq K_i\hspace{0.5cm}
	\text{ and } \hspace{0.5cm} K\pa{\frac{d_p(\xf,\X_i)}{h_{n,p}(\xf)}} \doteq K_{i,p}(\xf)\doteq
	K_{i,p},
	\]
	and as in (\ref{pesos_kernel}),
	\[
	\frac{K_i}{\sum_{j=1}^n K_j} \doteq W_i\hspace{0.5cm}
	\text{ and } \hspace{0.3cm} \frac{K_{i,p}}{\sum_{j=1}^n K_{j,p} }  \doteq W_{i,p}.
	\]
	Let us consider the following auxiliary unobservable quantities
	\[
	\etahat_n(\xf) = \sum_{i=1}^n W_iY_i, \hspace{0.3cm} \eta_n(\xf) = \sum_{i=1}^n W_i \eta(\X_i),
	\hspace{0.3cm}  \text{ and }  \hspace{0.3cm} \eta_{n,p}(\xf) = \sum_{i=1}^n W_{i,p}\eta(\X_i).
	\]
	Then we have,
	\begin{align*}
	\etahat_{n,p}(\xf) - \eta(\xf) &=  [\etahat_{n,p}(\xf) - \eta_{n,p}(\xf)] + [\eta_{n,p}(\xf) -
	\eta_n(\xf)] + [\eta_n(\xf) - \etahat_n(\xf)]  \\ &\hspace{1cm} + [\etahat_n(\xf) - \eta(\xf)]\\ &=  \sum_{i=1}^n W_{i,p} (Y_i- \eta(\X_i)) + \sum_{i=1}^n (W_{i,p}-W_i) \eta(\X_i) \\ &\hspace{1cm}+
	\sum_{i=1}^n W_i (\eta(\X_i)-Y_i) + [\etahat_n(\xf) - \eta(\xf)] \\
	&=   \sum_{i=1}^n (W_{i,p} - W_i) (Y_i- \eta(\X_i)) + \sum_{i=1}^n
	(W_{i,p}-W_i) \eta(\X_i) \\ &\hspace{1cm}+ [\etahat_n(\xf) - \eta(\xf)].
	\end{align*}
	Taking squares and expectation in $\Dn, \Cn$ we have
	\begin{align*}
	\esp{}{(\etahat_{n,p}(\xf) - \eta(\xf))^2)} &\lesssim
	\esp{}{\pa{ \sum_{i=1}^n (W_{i,p} - W_i) (Y_i-
			\eta(\X_i))}^2} \\& + \esp{}{\pa{\sum_{i=1}^n
			(W_{i,p}-W_i) \eta(\X_i)}^2} \\ &\hspace{1cm}+ \esp{}{\pa{[\etahat_n(\xf) -
			\eta(\xf)]}^2}\\ &\doteq I + II + III.
	\end{align*}
	By Proposition \ref{consistencia_kernel} and Remark \ref{h-mayor} (since $h_n (\xf) \to 0$
	and  $h_n (\xf) \ge h^*_n(\xf)$), taking expectation on $\X$ we have that term $III$ converges to
	zero. For the first term we have,
	\begin{align*}
	I &\approx \esp{}{\pa{ \sum_{i=1}^n (W_{i,p} - W_i)  (Y_i-
			\eta(\X_i))}^2} \nonumber \\ &=\esp{}{\sum_{i=1}^n \sum_{j=1}^n
		(W_{i,p} - W_i) (W_{j,p}  - W_j) e_i e_j} &  \pa{Y_i - \eta(\X_i) =
		e_i} \nonumber  \\&= \esp{}{\sum_{i=1}^n \sum_{j=1}^n
		(W_{i,p} - W_i) (W_{j,p}  - W_j) \esp{\Cn \vert \Dn}{\e_i
			\e_j\vert \Dn}}  \\&= \esp{}{\sum_{i=1}^n |W_{i,p} -
		W_i|^2 \esp{\Cn \vert \Dn}{\e_i^2 \vert \Dn}} &
	\pa{\text{cond. ind.}} \nonumber  \\&= \sigma^2
	\esp{}{\sum_{i=1}^n |W_{i,p} - W_i|^2}. \nonumber
	\end{align*}
	On the other hand, since $\eta$ is bounded, in $II$ we have
	\begin{equation*}
	II =  \esp{}{\pa{\sum_{i=1}^n (W_{i,p} - W_i) \eta(\X_i)}^2} \lesssim \esp{}{\pa{\sum_{i=1}^n
			\abs{W_{i,p} - W_i}}^2}.
	\end{equation*}
	We will see that terms $I$ and $II$ converge to zero by splitting the sum in different pieces:
	\begin{enumerate}
		\item[\textbf{(1)}] ${A_1 \doteq \{i: d_p(\xf,\X_i) > h_{n,p}(\xf),d(\xf,\X_i) > h_n(\xf)\}}$; 
		\item[\textbf{(2)}] ${A_2 \doteq \{i: d_p(\xf,\X_i) > h_{n,p}(\xf),d(\xf,\X_i)\le h_n(\xf)\}}$;
		\item[\textbf{(3)}] ${A_3 \doteq \{i: d_p(\xf,\X_i) \le h_{n,p}(\xf),d(\xf,\X_i) > 3h_n(\xf)\}}$;
		\item[\textbf{(4)}] ${A_4 \doteq \{i: d_p(\xf,\X_i) \le h_{n,p}(\xf),d(\xf,\X_i) \le 3 h_n(\xf)\}}$.
	\end{enumerate}
	Case \textbf{(1)} is trivial since in this case $K$ is supported in $[0,1]$ which impies that $W_{i,p} = W_i = 0$. Let us start, therefore, with case \textbf{(2)}.

	\begin{enumerate}
		\item[\textbf{(2)}] Let $\bm{A_2 \doteq \{i: d_p(\xf,\X_i) > h_{n,p}(\xf),d(\xf,\X_i)\le h_n(\xf)\}}$.  Observe that in this case $W_{i,p}
		= 0$ since $K$ is supported in $[0,1]$. Therefore, since $|W_i| \le 1$ we get 
		\[
		I_{A_2} \doteq \esp{}{\sum_{i=1}^n  |W_i|^2 \ind_{\{i\in A_2\}}} \le \esp{}{\sum_{i=1}^n
			\ind_{\{i\in A_2\}}},
		\]
		and,
		\begin{equation}\label{cotaA1}
		II_{A_2} \doteq \esp{}{\pa{\sum_{i=1}^n |W_i|\ind_{\{i\in A_2\}}}^2} \le \esp{}{\pa{\sum_{i=1}^n
				\ind_{\{i\in A_2\}}}^2} \doteq C_{A_2}. 
		\end{equation}
		Observe that the i.i.d. random variables $\ind_{\{i\in A_2\}}$ have a Bernoulli distribution with parameter
		\begin{align*}
		p &= \prob{\X_1}{d_p(\xf,\X_1) > h_{n,p}(\xf), d(\xf,\X_1) \le h_n(\xf)} \\
		&\le \prob{\X_1}{d_p(\xf,\X_1) - d(\xf,\X_1)  \ge h_{n,p}(\xf) - h_n(\xf)}  \\ &\le
		\prob{\X_1}{|d_p(\xf,\X_1) - d(\xf,\X_1)| \ge c_{n,p}}.  & \pa{\text{by H3.2}}
		\end{align*}
		As a consequence, the random variable $Z \doteq \sum_{i=1}^n \ind_{\{i\in A_2\}}$ has Binomial distribution with parameters $n$ and $p$ and expectation $\esp{}{Z} = np$. This implies that		
		\begin{equation}\label{IA1}
		I_{A_2}\lesssim \esp{}{Z} \le n \prob{\X_1}{|d_p(\xf,\X_1) - d(\xf,\X_1)| \ge c_{n,p}},
		\end{equation}
		and, since $ \esp{}{Z^2} = n p (1- p) + n^2 p^2 \le n p + (n p)^2$,
		\begin{align}\label{IIA1}
		II_{A_2} \le  C_{A_2} &\lesssim \esp{}{Z^2} \le n \prob{\X_1}{|d_p(\xf,\X_1) - d(\xf,\X_1)| \ge
			c_{n,p}}
		\\ &\hspace{3cm}+ (n\prob{\X_1}{|d_p(\xf,\X_1) - d(\xf,\X_1)| \ge c_{n,p}})^2\hspace{-0.1cm}. \nonumber
		\end{align}
		\item[\textbf{(3)}] Let $\bm{A_3 \doteq\{i: d_p(\xf,\X_i) \le h_{n,p}(\xf),d(\xf,\X_i) > 3h_n(\xf)\}}$. Observe that in this
		case $W_i = 0$ since $K$ is supported in $[0,1]$. Then, since $\forall \, i, \, \abs{W_{i,p}} \le 1$
		we get
		\begin{equation*}
		I_{A_3} \doteq \esp{}{\sum_{i=1}^n \abs{W_{i,p}}^2 \ind_{\{i \in A_3\}}} \le
		\esp{}{\sum_{i=1}^n \ind_{\{i \in A_3\}}},
		\end{equation*}
		and
		\begin{equation}\label{IIA21a}
		II_{A_3} \doteq \esp{}{\pa{\sum_{i=1}^n \abs{W_{i,p}}\ind_{\{i \in A_3\}}}^2} \le 
		\esp{}{\pa{\sum_{i=1}^n \ind_{\{i \in A_3\}}}^2}.
		\end{equation}
		Now, the i.i.d. random variables $\ind_{\{i \in A_3\}}$ have Bernoulli distribution with 
		parameter
		\begin{align*}
		p &= \prob{\X_1}{d_p(\xf,\X_1) \le h_{n,p}(\xf), d(\xf,\X_1) > 3 h_n(\xf)} \\ &\le
		\prob{\X_1}{d(\xf,\X_1) - d_p(\xf,\X_1) \ge 3 h_n(\xf) - h_{n,p}(\xf)}.
		\end{align*}
		As a consequence, the random variable $Z\doteq \sum_{i=1}^n
		\ind_{\{i \in A_3\}}$ has Binomial distribution
		with parameters $n$ and $p$. But from (\textit{H3.1}), for $n$ large enough, $h_n(\xf) \ge \pa{\frac{1+C_2}{2}}
		c_{n,p}$ which, together with \textit{H3.2} implies that
		\[
		3 h_n(\xf) - h_{n,p}(\xf) \ge 2 h_n(\xf) - C_2 c_{n,p}  \ge c_{n,p},
		\]
		and then, for $n$ large enough,
		\[
		p \le \prob{\X_1}{|d_p(\xf,\X_1) - d(\xf,\X_1)|  \ge c_{n,p}}.
		\]
		Therefore, since $\esp{}{Z} = n p$ we have
		\begin{equation}\label{IA21}
		I_{A_3}  \lesssim \esp{}{Z} \le n \prob{\X_1}{|d_p(\xf,\X_1) - d(\xf,\X_1) | \ge c_{n,p}},
		\end{equation}
		and since $\esp{}{Z^2} = n p (1-p) + n^2 p^2  \le n p + (n p)^2$,
		\begin{align}\label{IIA21}
		II_{A_3} &\lesssim \esp{}{Z^2} \le n \prob{\X_1}{|d_p(\xf,\X_1) - d(\xf,\X_1)| \ge c_{n,p}})
		\\ &\hspace{3cm}+ (n \prob{\X_1}{|d_p(\xf,\X_1) - d(\xf,\X_1)|  \ge c_{n,p}}))^2\hspace{-0.1cm}.  \nonumber
		\end{align}

		\item[\textbf{(4)}] Let $\bm{A_4 \doteq \{i: d_p(\xf,\X_i) \le h_{n,p}(\xf),d(\xf,\X_i) \le 3 h_n(\xf)\}}$.
		In this case we write,
		\begin{align*}
		W_{i,p} - W_i &= \frac{K_{i,p}}{\sum_{j=1}^n K_{j,p} }-\frac{K_i}{\sum_{j=1}^n
			K_j} \nonumber \\ &= \frac{K_{i,p}}{\sum_{j=1}^n K_{j,p} }-\frac{K_i}{\sum_{j=1}^n
			K_{j,p} }+\frac{K_i}{\sum_{j=1}^n K_{j,p} }-\frac{K_i}{\sum_{j=1}^n K_j} \\ &=
		(K_{i,p}-K_i) \frac{1}{\sum_{j=1}^n K_{j,p} } + K_i \frac{\sum_{j=1}^n (K_j - K_{j,p})}{\sum_{j=1}^n
			K_j \sum_{j=1}^n K_{j,p}}  \nonumber \\ &= (K_{i,p}-K_i) \frac{1}{\sum_{j=1}^n K_{j,p} } + W_i
		\frac{\sum_{j=1}^n (K_j - K_{j,p})}{\sum_{j=1}^n K_{j,p}}.\nonumber
		\end{align*}
		Then, 
		\begin{align}\label{IA22a}
		I_{A_4} &\doteq \esp{}{\sum_{i=1}^n |W_{i,p} - W_i|^2 \ind_{\{i \in A_4\}}}\nonumber \\
		&\lesssim \esp{}{\sum_{i=1}^n |K_{i,p}-K_i|^2
			\frac{\ind_{\{i \in A_4\}}}{(\sum_{j=1}^n
				K_{j,p})^2}}  \nonumber \\ &\hspace{0.3cm}+  \esp{}{\sum_{i=1}^n W_i^2 \ind_{\{i \in A_4\}}
			\pa{\frac{\sum_{j=1}^n (K_j - K_{j,p})}{\sum_{j=1}^n
					K_{j,p}}}^2}\\ &\lesssim \esp{}{\sum_{i=1}^n |K_{i,p}-K_i|^2 \frac{\ind_{\{i \in
					A_4\}}}{(\sum_{j=1}^n \ind_{\{j : d_p(\xf,\X_j) \le h_{n,p}(\xf) \}})^2}} & \pa{K \text{ regular}} \nonumber \\
		&\hspace{0.3cm}+ \esp{}{\pa{\frac{\sum_{j=1}^n |K_j
					- K_{j,p}|}{\sum_{j=1}^n K_{j,p}}}^2} & \hspace{-3cm}\pa{|W_i| \le 1, \sum_{i=1}^n W_i = 1} \nonumber \\ &\doteq I_{A_4}^1 + I_{A_4}^2,\nonumber\end{align}
		and,
		\begin{align}\label{IIA22a}
		II_{A_4} &\doteq \esp{}{\pa{\sum_{i=1}^n |W_{i,p} - W_i| \ind_{\{i \in A_4\}}}^2}\nonumber  \\
		&\lesssim \esp{}{\pa{\sum_{i=1}^n |K_{i,p}-K_i| \frac{\ind_{\{i \in A_4\}}}{\sum_{j=1}^n
					K_{j,p}}}^2} \nonumber \\ &\hspace{0.3cm}+ \esp{}{\pa{\sum_{i=1}^n W_i  \ind_{\{i \in A_4\}}
				\frac{\sum_{j=1}^n (K_j - K_{j,p})}{\sum_{j=1}^n K_{j,p}}}^2}\\ &\lesssim
		\esp{}{\pa{\sum_{i=1}^n |K_{i,p}-K_i|\frac{\ind_{\{i \in A_4\}}}{\sum_{j=1}^n \ind_{\{j : d_p(\xf,\X_j) \le h_{n,p}(\xf)\}}}}^2} &\hspace{-0.3cm} \pa{K \text{ regular}} \nonumber \\ &\hspace{0.3cm}+ \esp{}{\pa{\frac{\sum_{j=1}^n |K_j
					- K_{j,p}|}{\sum_{j=1}^n K_{j,p}}}^2} & \hspace{-0.5cm}\pa{|W_i| \le 1} \nonumber \\ &\doteq
		II_{A_4}^1 + II_{A_4}^2.\nonumber
		\end{align}
		Observe that if $\sum_{j=1}^n \ind_{\{j : d_p(\xf,\X_j) \le h_{n,p}(\xf)\}}= 0$ then $\forall \, j, \, \ind_{\{j \in
			A_4\}}=0$ so in this case, $I_{A_4}^1$ and $II_{A_4}^1$ are zero. Then, in
		what follows we will assume that $\sum_{j=1}^n \ind_{\{j : d_p(\xf,\X_j) \le h_{n,p}(\xf)\}} \neq 0$. Since $K$ is Lipschitz
		and we are  {only considering the indexes} $i$ such that $d_p (x,\X_i) \le h_{n,p}(x)$ we get,
		\begin{align*}
		|K_{i,p}-K_i| &= \abs{K\pa{\frac{d_p(\xf,\X_i)}{h_{n,p}(\xf)}}-K\pa{\frac{d(\xf,\X_i)}{h_n(\xf)}}}
		\\&\lesssim \abs{\frac{d_p(\xf,\X_i)}{h_{n,p}(\xf)}-\frac{d(\xf,\X_i)}{h_n(\xf)}} \\ &=
		\frac{\abs{d_p(\xf,\X_i)h_n(\xf)-d(\xf,\X_i)h_{n,p}(\xf)}}{h_{n,p}(\xf)h_n(\xf)}  \\ &\le
		\frac{\abs{d_p(\xf,\X_i)-d(\xf,\X_i)}}{h_n(\xf)} + \frac{d_p (x,\X_i) |h_{n} (x) -
			h_{n,p}(x)|}{h_n(x) h_{n,p}(x)}\\ & \lesssim \frac{\abs{d_p(\xf,\X_i)-d(\xf,\X_i)}}{h_n(\xf)} +
		\frac{c_{n,p}}{h_n(x)}. &\hspace{-1cm} \pa{\text{by  H3.2}}
		\end{align*}
		Therefore, 
		\begin{align} \label{IA221}
		I_{A_4}^1 &\lesssim \frac{1}{h^2_n(\xf)} \esp{}{\sum_{i=1}^n |d_p(\xf,\X_i)-d (\xf,\X_i)|^2
			\frac{\ind_{\{i \in A_4\}}}{(\sum_{j=1}^n
				\ind_{\{j : d_p(\xf,\X_j) \le h_{n,p}(\xf)\}})^2}}  \nonumber \\ &\hspace{1cm}+
		\left(\frac{c_{n,p}}{h_n(x)}\right)^2 \esp{}{\sum_{i=1}^n\frac{\ind_{\{i \in
					A_4\}}}{(\sum_{j=1}^n \ind_{\{j : d_p(\xf,\X_j) \le h_{n,p}(\xf)\}})^2}} \\ &\lesssim \frac{1}{h^2_n(\xf)}
		\esp{}{\sum_{i=1}^n |d_p(\xf,\X_i)-d(\xf,\X_i)|^2 \frac{\ind_{\{j \in A_4\}}}{(\sum_{j=1}^n \ind_{\{j : d_p(\xf,\X_j) \le h_{n,p}(\xf)\}})^2}} \nonumber\\ &\hspace{1cm}+ \left(\frac{c_{n,p}}{h_n(x)}\right)^2\hspace{-0.1cm}, \nonumber
		\end{align}
		and,
		\begin{align}\label{IIA221}
		II_{A_4}^1 &\lesssim \frac{1}{h^2_n(\xf)}  \esp{}{\pa{\sum_{i=1}^n |d_p(\xf,\X_i)-d
				(\xf,\X_i)| \frac{\ind_{\{i \in A_4\}}}{\sum_{j=1}^n \ind_{\{j : d_p(\xf,\X_j) \le h_{n,p}(\xf)\}}}}^2} \nonumber \\
		&\hspace{1cm}+
		\left(\frac{c_{n,p}}{h_n(x)}\right)^2 \esp{}{\pa{\sum_{i=1}^n \frac{\ind_{\{i \in
						A_4\}}}{\sum_{j=1}^n \ind_{\{j : d_p(\xf,\X_j) \le h_{n,p}(\xf)\}}}}^2}\\  &\lesssim \frac{1}{h^2_n(\xf)}
		\esp{}{\pa{\sum_{i=1}^n
				|d_p(\xf,\X_i)-d (\xf,\X_i)| \frac{\ind_{\{i \in A_4\}}}{\sum_{j=1}^n \ind_{\{j : d_p(\xf,\X_j) \le h_{n,p}(\xf)\}}}}^2} \nonumber\\ &\hspace{1cm}+ \left(\frac{c_{n,p}}{h_n(x)}\right)^2\hspace{-0.1cm}. \nonumber
		\end{align}

		\
		
		\textbf{(4.1)} Let $\bm{A_{41} \doteq A_4 \cap \{i: |d_p(\xf,\X_i)-d (\xf,\X_i)|\le c_{n,p}\}}$.
		In this case, by (\textit{H3.1}) we get
		\begin{align}\label{IA211}
		I_{A_{41}}^1 &\doteq \frac{c_{n,p}^2}{h^2_n(\xf)}\esp{}{\frac{\sum_{i=1}^n
				\ind_{\{i \in A_4\}}}{(\sum_{j=1}^n \ind_{\{j : d_p(\xf,\X_j) \le h_{n,p}(\xf)\}})^2} } +
		\left(\frac{c_{n,p}}{h_n(x)}\right)^2 \\& \lesssim \pa{\frac{c_{n,p}}{h_n(\xf)}}^2\hspace{-0.1cm}, \nonumber
		\end{align}
		and
		\begin{align}\label{IIA211}
		II_{A_{41}}^1 &\doteq  \frac{c_{n,p}^2}{h^2_n(\xf)}
		\esp{}{\pa{\frac{\sum_{i=1}^n\ind_{\{i \in A_4\}} }{\sum_{j=1}^n \ind_{\{j : d_p(\xf,\X_j) \le h_{n,p}(\xf)\}}}}^2}  +
		\left(\frac{c_{n,p}}{h_n(x)}\right)^2 \\ &\lesssim \pa{\frac{c_{n,p}}{h_n(\xf)}}^2\hspace{-0.2cm} . \nonumber
		\end{align}
		
		\
		
		\textbf{(4.2)} Let $\bm{A_{42} \doteq A_4 \cap \{i:|d_p(\xf,\X_i)-d (\xf,\X_i)| > c_{n,p}\}} 
		$. Let us define the i.i.d. random variables $Z_i \doteq d_p(\xf,\X_i)-d (\xf,\X_i)$, $i = 1,\ldots,
		n$.
		Since $d_p(\xf,\X_i)\le h_{n,p}(\xf)$ and $d(\xf,\X_i) \le  3h_n(\xf)$ we
		have that $|Z_i| \le h_{n,p}(\xf) + 3h_n(\xf)$. Observe that, from (\textit{H3.2}) and
		(\textit{H3.1}), respectively, for $n$ large enough we have
		\[
		h_{n,p} \le h_n(\xf) + C_2c_{n,p} \le C h_n(\xf).
		\]
		Which implies that, for $n$ large enough, $|Z_i| \le C h_n(\xf)$. Therefore,
		\begin{align}\label{I212}
		I_{A_{42}}^1 &\doteq  \frac{1}{h^2_n(\xf)} \esp{}{\sum_{i = 1}^n |Z_i|^2  \ind_{\{i:c_{n,p} \le |Z_i|\le C h_n(\xf)\}} }  \nonumber \\ &\hspace{0.5cm} + \left(\frac{c_{n,p}}{h_n(x)}\right)^2 \\ &\le \frac{1}{h^2_n(\xf)} \esp{}{\sum_{i = 1}^n 
			|Z_i|^2  \ind_{\{i:c_{n,p} \le |Z_i|\le C h_n(\xf)\}}} \nonumber \\ &\hspace{0.5cm} +
		\left(\frac{c_{n,p}}{h_n(x)}\right)^2   \nonumber \\ &\le \frac{n}{h^2_n(\xf)} \esp{}{|Z_1|^2
			\ind_{\{c_{n,p} \le |Z_1|\le C h_n(\xf)\}}}+ \left(\frac{c_{n,p}}{h_n(x)}\right)^2
		&\pa{\#A_{42} \le n} \nonumber\\ &\lesssim \frac{n}{h_n(\xf)}
		\esp{}{|Z_1|\ind_{\{c_{n,p} \le |Z_1|\le C h_n(\xf)\}}}+ \left(\frac{c_{n,p}}{h_n(x)}\right)^2\hspace{-0.1cm}.
		&\pa{|Z_1| \lesssim h_n(\xf)} \nonumber
		\end{align}
		On the other hand,
		\begin{align}\label{II212}
		II_{A_{42}}^1 &\doteq  \frac{1}{h^2_n(\xf)} \esp{}{\pa{\sum_{i = 1}^n |Z_i|\ind_{\{i:c_{n,p} \le |Z_i|\le C h_n(\xf)\}}}^2} \nonumber  \\ &\hspace{0.5cm}+ \left(\frac{c_{n,p}}{h_n(x)}\right)^2  \\&\le \frac{1}{h^2_n(\xf)}
		\esp{}{\pa{\sum_{i=1}^n |Z_i|\ind_{\{i:c_{n,p} \le |Z_i|\le C h_n(\xf)\}}}^2} \nonumber 
		\\ &\hspace{0.5cm}+\left(\frac{c_{n,p}}{h_n(x)}\right)^2\hspace{-0.1cm}. \nonumber
		\end{align}
		Observe that, for $i\neq j$, $Z_i$ is independent of $Z_j$ then,
		\begin{align*}
		&\esp{}{\pa{\sum_{i = 1}^n |Z_i|\ind_{\{i:c_{n,p} \le |Z_i|\le C h_n(\xf)\}}}^2} \\&= 
		\esp{}{\sum_{i = 1}^n \sum_{j=1}^n  |Z_i||Z_j| \ind_{\{i:c_{n,p} \le |Z_i|\le C h_n(\xf)\}}
			\ind_{\{j:c_{n,p} \le |Z_j|\le C h_n(\xf)\}}}\\ &= \esp{}{\sum_{i = 1}^n |Z_i|^2
			\ind_{\{i:c_{n,p} \le |Z_i|\le C h_n(\xf)\}}} \\ &\hspace{0.3cm} + \mathbb E \Bigg(\sum_{i = 1}^n
		\sum_{\substack{j=1 \\ j\neq i}}^n |Z_i||Z_j| \ind_{\{i:c_{n,p} \le |Z_i|\le C
			h_n(\xf)\}}\ind_{\{j:c_{n,p} \le |Z_j|\le C h_n(\xf)\}}\Bigg) \\ &\le n \esp{}{|Z_1|^2
			\ind_{\{c_{n,p} \le |Z_1|\le C h_n(\xf)\}}} \\
		&\hspace{0.3cm} + n^2 \esp{}{|Z_1|\ind_{\{c_{n,p} \le |Z_1|\le C h_n(\xf)\}}} \esp{}{|Z_1|
			\ind_{\{c_{n,p} \le |Z_1|\le C h_n(\xf)\}}} \\ &\lesssim n h_n(x)
		\esp{}{|Z_1|\ind_{\{c_{n,p} \le |Z_1|\le C h_n(\xf)\}}} &\hspace{-1.5cm}\pa{|Z_1| \lesssim h_n(\xf)}
		\\ &\hspace{0.3cm} + n^2 \pa{\esp{}{|Z_1|\ind_{\{c_{n,p} \le |Z_1|\le C h_n(\xf)\}}}}^2\hspace{-0.1cm}.
		\end{align*}
		{Using this bound  in} (\ref{II212}) we get,
		\begin{align}\label{II212-1}
		II_{A_{42}}^1 &\lesssim \frac{n}{h_n(\xf)} \esp{}{|Z_1|\ind_{\{c_{n,p} \le |Z_1|\le C h_n(\xf)\}}}
		\\
		&\hspace{0.5cm}+ \frac{n^2}{h_n^2(\xf)} \pa{\esp{}{|Z_1|\ind_{\{c_{n,p} \le |Z_1|\le C
					h_n(\xf)\}}}}^2+ \left(\frac{c_{n,p}}{h_n(x)}\right)^2\hspace{-0.1cm}. \nonumber
		\end{align}
		We need to compute the expectation
		$\esp{}{|Z_1|\ind_{\{c_{n,p} \le|Z_1|\le C h_n(\xf)\}}}$ which is,
		\begin{align*}
		\esp{}{|Z_1|\ind_{\{c_{n,p} \le|Z_1|\le C h_n(\xf)\}}} &= \int_{c_{n,p}}^{h_n(\xf)} \prob{}{|Z_1| >
			t } \, dt \\ &\le  \prob{}{|Z_1| > c_{n,p}}  \int_{c_{n,p}}^{h_n(\xf)}\, dt \\ &\le\prob{}{|Z_1| >
			c_{n,p}} h_n(\xf).
		\end{align*}
		Therefore, with this inequality in (\ref{I212}) we have
		\begin{align}\label{IA212}
		I_{A_{42}}^1 &\lesssim n \prob{}{|Z_1| > c_{n,p}} + \left(\frac{c_{n,p}}{h_n(x)}\right)^2
		\\ &=  n \prob{}{|d_p(\xf,\X_1)-d (\xf,\X_1)| > c_{n,p}} +
		\left(\frac{c_{n,p}}{h_n(x)}\right)^2\hspace{-0.1cm}, \nonumber 
		\end{align}
		and,  with the same inequality in (\ref{II212-1}), 
		\begin{align}\label{IIA212}
		II_{A_{42}}^1 &\lesssim  n \prob{}{|Z_1| > c_{n,p}} + ( n \prob{}{|Z_1| > c_{n,p}} )^2 +
		\left(\frac{c_{n,p}}{h_n(x)}\right)^2  \\ &= n \prob{}{|d_p(\xf,\X_1)-d (\xf,\X_1)| >
			c_{n,p}} \nonumber \\ &\hspace{0.5cm}+ ( n \prob{}{|d_p(\xf,\X_1)-d (\xf,\X_1)| > c_{n,p}} )^2 + \left(\frac{c_{n,p}}{h_n(x)}\right)^2\hspace{-0.1cm}. \nonumber
		\end{align}
		
		\
		
		Then, with (\ref{IA211}) and (\ref{IA212}) in (\ref{IA221}) we get
		\begin{equation}\label{IA22.1}
		I_{A_4}^1 \lesssim \pa{\frac{c_{n,p}}{h_n(\xf)}}^2 + n \prob{}{|d_p(\xf,\X_1)-d (\xf,\X_1)| >
			c_{n,p}}.
		\end{equation}
		and, with (\ref{IIA211}) and (\ref{IIA212}) in (\ref{IIA221}),
		\begin{align}\label{IIA22.1}
		II_{A_4}^1 &\lesssim \pa{\frac{c_{n,p}}{h_n(\xf)}}^2 + n \prob{}{|d_p(\xf,\X_1)-d (\xf,\X_1)| >
			c_{n,p}} \\ &\hspace{2cm}+ ( n \prob{}{|d_p(\xf,\X_1)-d (\xf,\X_1)| > c_{n,p}} )^2\hspace{-0.1cm}. \nonumber
		\end{align}
		On the other hand, observe that $I_{A_4}^2  = \esp{}{\pa{\frac{\sum_{j=1}^n |K_j
					- K_{j,p}|}{\sum_{j=1}^n K_{j,p}}}^2}$. Since $A_4^c = \{j: d(\xf,\X_j) > 3h_n(\xf)\}
		\cup \{j: d_p(\xf,\X_j) >  h_{n,p}(\xf)\}$ we can write,
		\begin{align*}
		\frac{\sum_{j=1}^n |K_j - K_{j,p}|}{\sum_{j=1}^n K_{j,p}} &\le \frac{\sum_{j=1}^n |K_j
			- K_{j,p}| \ind_{\{j \in A_4\}}}{\sum_{j=1}^n K_{j,p}} \\ &\hspace{1cm} + \frac{\sum_{j=1}^n |K_j
			- K_{j,p}|\ind_{\{j: d(\xf,\X_j) >  3h_n(\xf)\}}}{\sum_{j=1}^n K_{j,p}} \\
		&\hspace{1cm}+\frac{\sum_{j=1}^n |K_j - K_{j,p}|\ind_{\{j: d_p(\xf,\X_j) > 
				h_{n,p}(\xf)\}}}{\sum_{j=1}^n K_{j,p}}.
		\end{align*}
		Using that $K$ is regular and that $\sum_{j=1}^n K_{j,p} \ge 1$  {(this is since $\{j : d_p(\xf,\X_j) \le h_{n,p}(\xf)\} \neq \emptyset$) }
		we get,
		\begin{align*}
		I_{A_4}^2 &=  \esp{}{\pa{\frac{\sum_{j=1}^n |K_j - K_{j,p}|}{\sum_{j=1}^n
					K_{j,p}}}^2} \nonumber \\&\lesssim  II_{A_4}^1 +  \esp{}{\pa{\sum_{j=1}^n |W_{j,p}| \ind_{\{j:
					d_p(\xf,\X_j) \le h_{n,p}(\xf), d(\xf,\X_j) >  3h_n(\xf)\}}}^2} \nonumber \\
		&\hspace{1cm}+\frac{\sum_{j=1}^n K_j \ind_{\{j: d_p(\xf,\X_j) > h_{n,p}(\xf)\}}}{\sum_{j=1}^n
			K_{j,p}}  \\ &\lesssim II_{A_4}^1 + II_{A_3} 
		+ \esp{}{\pa{\sum_{j=1}^n \ind_{\{j: d_p(\xf,\X_j) > h_{n,p}(\xf), d(\xf,\X_j) \le h_n(\xf)\}}}^2}
		\nonumber \\ &\le II_{A_4}^1 + II_{A_3} + C_{A_2}, 
		\end{align*}
		where $II_{A_4}^1$ was defined in (\ref{IIA22a}), $II_{A_3}$ in (\ref{IIA21a}), and
		$C_{A_2}$ in (\ref{cotaA1}). Then, from (\ref{IIA22.1}), (\ref{IIA21}), and (\ref{IIA1}) we have
		\begin{align}\label{IA22.2-IIA22.2}
		I_{A_4}^2 &\lesssim \pa{\frac{c_{n,p}}{h_n(\xf)}}^2 + n \prob{}{|d_p(\xf,\X_1)-d (\xf,\X_1)| > c_{n,p}}  \\ &\hspace{1cm}+ ( n \prob{}{|d_p(\xf,\X_1)-d (\xf,\X_1)| >
			c_{n,p}} )^2\hspace{-0.1cm}. \nonumber
		\end{align}
		Therefore, with (\ref{IA22.1}) and (\ref{IA22.2-IIA22.2}) in (\ref{IA22a}) we have
		\begin{align}\label{IA22}
		I_{A_4} &\lesssim \pa{\frac{c_{n,p}}{h_n(\xf)}}^2 + n \prob{}{|d_p(\xf,\X_1)-d (\xf,\X_1)| >
			c_{n,p}} \\ &\hspace{3cm}+ ( n \prob{}{|d_p(\xf,\X_1)-d (\xf,\X_1)| > c_{n,p}} )^2\hspace{-0.1cm},  \nonumber 
		\end{align}
		and with (\ref{IIA22.1}) and (\ref{IA22.2-IIA22.2}) in  in (\ref{IIA22a}),
		\begin{align}\label{IIA22}
		II_{A_4} &\lesssim \pa{\frac{c_{n,p}}{h_n(\xf)}}^2 + n \prob{}{|d_p(\xf,\X_1)-d (\xf,\X_1)| >  c_{n,p}}\\ &\hspace{3cm}+ ( n \prob{}{|d_p(\xf,\X_1)-d (\xf,\X_1)| > c_{n,p}} )^2\hspace{-0.1cm}.  \nonumber 
		\end{align}
	\end{enumerate}
	Finally, to complete the proof of this result (i.e. that $I$ and $II$ converge to zero) we need to
	show
	that the expectation on $\X$ of
	\begin{align*}
	& \pa{\frac{c_{n,p}}{h_n(\xf)}}^2 + n \prob{\X_1}{|d_p(\xf,\X_1)-d (\xf,\X_1)| > c_{n,p}} \nonumber \\& \hspace{4cm} + ( n \mathbb{P}^2_{\X_1}{|d_p(\xf,\X_1)-d (\xf,\X_1)| > c_{n,p}}),
	\end{align*}
	{converges  to zero. In order to show it,} recall that from \textit{H2} we have
	\[
	n^2 \esp{\X}{\mathbb{P}^2_{\X_1|\X}{\big(|d_p(\X,\X_1) - d(\X,\X_1)|  \ge c_{n,p}}\big) \big|\X \in 
		\sop{\mu}} \to 0,
	\]
	and consequently, by Cauchy Schwartz inequality
	\[
	n \esp{\X}{\prob{\X_1|\X}{|d_p(\X,\X_1) - d(\X,\X_1)|  \ge c_{n,p}}) \big|\X \in \sop{\mu}} \to 0.
	\]
	In addition from (\textit{H3.1}) we have,
	\[
	\esp{\X}{\pa{\frac{c_{n,p}}{h_n(\X)}}^2} \to 0.
	\]
	Therefore, taking expectation with respect to $\X$ in (\ref{IA1}), (\ref{IIA1}), (\ref{IA21}),
	(\ref{IIA21}), (\ref{IA22}), and, (\ref{IIA22}), we prove \textit{Part (a)} of the Theorem.
	
	\textit{Proof of (b):} The only difference with \textit{item (a)} is the convergence of
	term $III$ to zero which is ensured by Proposition \ref{consistencia_knn-k}.
\end{proof}

\begin{proof}[Proof of Theorem \ref{teo4}]
	Let $\gamma_n \to \infty$ as $n\to \infty$ a sequence such that, as $n,p\to \infty$, 
	$\esp{\X}{\gamma_n
		\left(\frac{c_{n,p}}{h_n(\X)}\right)^2} \to 0$ and, for each $i=1,\ldots,n$,	
		$$\gamma_n
	n^2\esp{\X}{\mathbb{P}^2_{\X_i|\X}{\big(\abs{d(\X,\X_i)-d_p(\X, \X_i)} \ge c_{n,p} \Big|\X \in 
			\sop{\mu}}\big)} \to 0.$$ From proof of Theorem \ref{teo3} we get,
	\begin{align*}
	\esp{}{\gamma_n (\etahat_{n,p}(\X) - \eta(\X))^2} &\lesssim \gamma_n n
	\esp{\X}{\prob{\X_1}{d_p(\xf,\X_1) - d(\xf,\X_1) \ge c_{n,p}}} \\
	&\hspace{0.5cm}+ \esp{\X}{\gamma_n \left(\frac{c_{n,p}}{h_n(\X)}\right)^2} \\
	&\hspace{0.5cm}+ 
	\esp{}{\gamma_n(\etahat_n(\X)  - \eta(\X))^2},
	\end{align*}
	from what follows that,
	\[
	\lim_{n,p \to \infty}  \esp{}{\gamma_n (\etahat_{n,p}(\X) - \eta(\X))^2}
	= 0.
	\]
\end{proof}

%

\markboth{\hfill{\footnotesize\rm Liliana Forzani, Ricardo Fraiman and Pamela Llop} \hfill}
{\hfill {\footnotesize\rm Functional Nonparametric Regression} \hfill}

\bibhang=1.7pc
\bibsep=2pt
\fontsize{9}{14pt plus.8pt minus .6pt}\selectfont
\renewcommand\bibname

\end{document}